\documentclass[preprint]{elsarticle}

\usepackage{amsmath}
\usepackage{amssymb}
\usepackage{amsthm}
\usepackage{amscd}

\input{xy}
\xyoption{all}

\newtheorem{theo}{Theorem}[section]
\newtheorem{lem}[theo]{Lemma}
\newtheorem{stat}[theo]{Proposition}
\newtheorem{cons}[theo]{Corollary}

\theoremstyle{definition}

\theoremstyle{remark}
\newtheorem{rem}[theo]{Remark}
\newtheorem*{rem*}{Remark}
\newtheorem{que}[theo]{Question}

\let\ge\geqslant
\let\le\leqslant
\let\emptyset\varnothing
\let\ES\varnothing

\let\eps\varepsilon

\let\phi\varphi
\let\kappa\varkappa
\let\hra\hookrightarrow

\newcommand{\Cl}{\mathop{\mathrm{Cl}}}

\newcommand{\supp}{\mathop{\mathrm{supp}}}
\newcommand{\sub}{\mathop{\mathrm{sub}}}

\newcommand{\pr}{\mathop{\mathrm{pr}}}

\newcommand{\expl}{\mathop{\mathrm{exp}_{\,l}}}
\newcommand{\expu}{\mathop{\mathrm{exp}_{\,u}}}
\newcommand{\expus}{\mathop{\mathrm{exp}_{\,us}}}

\newcommand{\uni}[1]{\mathbf{1}_{{#1}}}
\newcommand{\Ob}{\mathop{\mathrm{Ob}}}

\newcommand{\Comp}{\mathcal{C}\mathrm{omp}}

\newcommand{\Tych}{\mathcal{T}\mathrm{ych}}
\newcommand\subcl{\mathrel{\underset{\mathrm{cl}}{\subset}}}
\newcommand\subop{\mathrel{\underset{\mathrm{op}}{\subset}}}

\newcommand{\jint}{\sideset{}{^\lor}\int\limits}
\newcommand{\BBN}{\mathbb{N}}

\newcommand{\BBR}{\mathbb{R}}

\newcommand{\BBF}{\mathbb{F}}
\newcommand{\BBG}{\mathbb{G}}
\newcommand{\BBM}{\mathbb{M}}
\newcommand{\BBP}{\mathbb{P}}

\newcommand{\CCA}{\mathcal{A}}
\newcommand{\CCB}{\mathcal{B}}
\newcommand{\CCC}{\mathcal{C}}

\newcommand{\CCF}{\mathcal{F}}
\newcommand{\CCG}{\mathcal{G}}
\newcommand{\CCH}{\mathcal{H}}
\newcommand{\CCI}{\mathcal{I}}

\newcommand{\CCU}{\mathcal{U}}
\newcommand{\CCV}{\mathcal{V}}

\begin{document}

\begin{frontmatter}



\title{Inclusion hyperspaces and capacities on Tychonoff spaces: functors
and monads\tnoteref{t1}}

\tnotetext[t1]{This research was supported by the Slovenian Research
Agency grants P1-0292-0101, J1-9643-0101 and BI-UA/09-10-002, and
by the Ministry of Science and Education of Ukraine project
M/113-2009.}


\author[prec]{Oleh Nykyforchyn\corref{cor}}
\ead{oleh.nyk@gmail.com}

\cortext[cor]{Corresponding author}

\author[ul]{Du\v san Repov\v s}
\ead{dusan.repovs@guest.arnes.si}

\address[prec]{Department of Mathematics and Computer Science,
Vasyl' Stefanyk Precarpathian National University, Shevchenka 57,
Ivano-Frankivsk, Ukraine, 76025}

\address[ul]{Faculty of Mathematics and Physics, and Faculty of Education,
University of Ljubljana, P.O. Box 2964, Ljubljana, Slovenia, 1001}

\begin{abstract}
The inclusion hyperspace functor, the capacity functor and monads
for these functors have been extended from the category of compact
Hausdorff spaces to the category of Tychonoff spaces. Properties of
spaces and maps of inclusion hyperspaces and capacities
(non-additive measures) on Tychonoff spaces are investigated.
\end{abstract}


\end{frontmatter}

\section*{Introduction}

The category of compact Hausdorff topological spaces is probably
the most convenient topological category for a categorical
topologist. A situation is usual when some results are first
obtained for compacta and then extended with much effort to a wider
class of spaces and maps, see e.g. factorization theorems for
inverse limits~\cite{Sh:FuncUncPowComp:81}. Many classical
construction on topological spaces lead to covariant functors in
the category of compacta, and categorical methods proved to be
efficient tools to study hyperspaces, spaces of measures, symmetric
products etc~\cite{ZarFed:CovFuncTopCat:90}. We can mention the
hyperspace functor $\exp$~\cite{TZ:CTCHS:99}, the inclusion
hyperspace functor $G$~\cite{Mo:InclHyp:88}, the probability
measure functor $P$~\cite{Fed:FuncProbMeasTopCat:98}, and the
capacity functor $M$ which was recently introduced by Zarichnyi and
Nykyforchyn \cite{ZN:CapFun:08} to study non-additive regular
measures on compacta.

Functors $\exp$, $P$, $G$, $M$ have rather good properties. The
functors $\exp$ and $P$ belong to a defined by \v S\v cepin class
of normal functors, while $G$ and $M$ satisfy all requirements of
normality but preservation of preimages, hence are only weakly
normal. They are functorial parts of
monads~\cite{TZ:CTCHS:99,ZN:CapFun:08}.

Unfortunately the functors $\exp$ and $G$ lose most of their nice
properties when they are extended from the category of compacta to
the category of Tychonoff spaces. Moreover, a~meaningful extension
usually is not unique. An interested reader is referred, e.g. to
\cite{Ban:TopSpProbMeas-I:95}, where \emph{four} extensions to the
category of Tychonoff spaces of the probability measure functor $P$
are discussed, and two of them are investigated in detail.

The aim of this paper is extend the inclusion hyperspace functor,
the capacity functor and monads for these functors from the
category of compacta to the category of Tychonoff spaces, and to
study properties of these extensions. We will use ``fine tuning''
of standard definitions of hyperspaces and inclusion hyperspaces to
``save'' as much topological and categorical properties valid for
the compact case as possible.

\section{Preliminaries}

In the sequel a \emph{compactum} is a compact Hausdorff topological
space. The~\emph{unit segment} $I=[0;1]$ is considered as a
subspace of the real line $\BBR$ with the natural topology. We say
that a function $\phi:X\to I$ \emph{separates} subsets $A,B\subset
X$ if $\phi|_A\equiv 1$, $\phi|_B\equiv 0$. If such $\phi$ exists
for $A$ and $B$ and is continuous, then we call these sets
\emph{completely separated}. We write $A\subop X$ or $A\subcl X$
if $A$ is respectively an open or a closed subset of a space $X$.
The set of all continuous functions from a space $X$ to a space $Y$
is denoted by $C(X,Y)$.

See \cite{ML:CWM:98} for definitions of category, functor, natural
transformation, monad (triple), morphism of monads. For a category
$\CCC$ we denote the class of its objects by $\Ob\CCC$. The
{category of Tychonoff spaces} $\Tych$ consists of all Tychonoff (=
completely regular) spaces and continuous maps between them. The
\emph{category of compacta} $\Comp$ is a full subcategory of
$\Tych$ and contains all compacta and their continuous maps. We say
that a~functor $F_1$ in $\Tych$ or in $\Comp$ is a
\emph{subfunctor} of a~functor $F_2$ in the same category if there is
a natural transformation $F_1\to F_2$ with all components being
embeddings. Similarly a~monad $\BBF_1$ is a \emph{submonad} of
a~monad $\BBF_2$ if there is a morphism of monads $\BBF_1\to\BBF_2$
such that all its components are embeddings.

>From now on we denote the set of all nonempty closed subsets of a
topological space $X$ by $\exp X$, though sometimes this notation
is used for the set of all \emph{compact} non-empty subsets, and
the two meaning can even coexist in one text~\cite{FF:GenTop:88}. A
lot of topologies on $\exp X$ can be found in literature. The
\emph{upper topology} $\tau_u$ is generated by the base which
consist of all sets $\{F\in \exp X\mid F\subset U\}$, where $U$ is
open in $X$. The
\emph{lower topology} $\tau_l$ has the subbase $\{\{F\in\exp X\mid
F\cap X\ne\ES\}\mid U\subop X\}$. The \emph{Vietoris topology}
$\tau_v$ is the least topology that contains both the upper and the
lower topologies. It is \emph{de facto} the default topology on
$\exp X$, to the great extent due to an important fact that, for a
compact Hausdorff space $X$, the space $\exp X$ with the Vietoris
topology is compact and Hausdorff. It $f:X\to Y$ is a continuous
map of compacta, then the map $\exp f:\exp X\to
\exp Y$, which sends each non-empty closed subset $F$ of $X$ to its
image $f(Y)$, is continuous. Thus we obtain the~\emph{hyperspace
functor} $\exp:\Comp\to \Comp$.

A non-empty closed with respect to the Vietoris topology subset
$\CCF\subset\exp X$ is called an~\emph{inclusion hyperspace} if
$A\subset B\in \exp X$, $A\in \CCF$ imply $B\in \CCF$. The set $GX$
of all inclusion hyperspaces on the space $X$ is closed in
$\exp^2X$, hence is a compactum with the induced topology if $X$ is
a compactum. This topology can also be determined by a subbase
which consists of all sets of the form
\begin{gather*}
U^+=
\{\CCF\in GX\mid \text{there is }F\in\CCF,
F\subset U\},
\\
U^-=
\{\CCF\in GX\mid F\cap U\ne\ES
\text{ for all }F\in\CCF\},
\end{gather*}
with $U$ open in $X$. If the map $Gf:GX\to GY$ for a continuous map
$f:X\to Y$ of compacta is defined as $Gf(G)=\{B\subcl Y\mid
B\supset f(A)\text{ for some }A\in \CCF\}$, $\CCF\in GX$, then $G$
is the~\emph{inclusion hyperspace functor} in $\Comp$.

We follow a terminology of~\cite{ZN:CapFun:08} and call a function
$c:\exp X\cup\{\emptyset\}\to I$ a \emph{capacity} on a compactum
$X$ if the three following properties hold for all closed subsets
$F$, $G$ of~$X$~:
\begin{enumerate}
\item
$c(\emptyset)=0$, $c(X)=1$;
\item
if $F\subset G$, then $c(F)\le c(G)$ (monotonicity);
\item
if $c(F)<a$, then there exists an open set $U\supset F$ such that
for any $G\subset U$ we have $c(G)<a$ (upper semicontinuity).
\end{enumerate}
The set of all capacities on a compactum $X$ is denoted by $MX$. It
was shown in \cite{ZN:CapFun:08} that a compact Hausdorff topology
is determined on $MX$ with a subbase which consists of all sets of
the form
$$
O_-(F,a)=\{c\in MX\mid c(F)<a\},
$$
where $F\subcl X$, $a\in \BBR$, and
\begin{multline*}
O_+(U,a)=\{c\in MX\mid c(U)>a\}=\\
\{c\in MX\mid \text{there exists a compactum } F\subset U, c(F)>a\},
\end{multline*}
where $U\subop X$, $a\in \BBR$. The same topology can be defined as
weak${}^*$ topology, i.e. the weakest topology on $MX$ such that
for each continuous function $\phi:X\to [0;+\infty)$ the
correspondence which sends each $c\in MX$ to the \emph{Choquet
integral}~\cite{Ch:TheoCap:53} of $\phi$ w.r.t.~$c$
$$
\int_X\phi(x)\,dc(x)=\int_0^{+\infty}c(\{x\in X\mid \phi(x)\ge
a\})\,da
$$
is continuous. If $f:X\to Y$ is a continuous map of compacta, then
the map $Mf:MX\to MY$ is defined as follows~:
$Mf(c)(F)=c(f^{-1}(F))$, for $c\in MX$ and $F\subcl Y$. This map is
continuous, and we obtain the~\emph{capacity functor} $M$ in the
category of compacta.

A monad $\mathbb{F}$ in a category $\CCC$ is a triple
$(F,\eta_F,\mu_F)$, with $F:\CCC\to \CCC$ a functor,
$\eta_F:\uni{\CCC}\to F$ and $\mu_F:F^2\to F$ natural
transformations, such that $\mu_FX\circ \eta_F FX=\mu_FX\circ
F\eta_FX=\uni{FX}$, $\mu_FX\circ F\mu_F X=\mu_FX\circ \mu_F FX$ for
all objects $X$ of~$\CCC$. Then $F,\eta_F,\mu_F$ are called resp.\
the \emph{functorial part}, the \emph{unit} and the
\emph{multiplication} of $\mathbf{F}$. For the~\emph{inclusion
hyperspace monad} $\BBG=(G,\eta_G,\mu_G)$ the components of the
unit and the multiplication are defined by the
formulae~\cite{Rad:MonGalg:90}:
$$
\eta_GX(x)=\{F\in\exp X\mid F\ni x\},\; x\in X,
$$
and
$$
\mu_GX(\mathbf{F})=\{F\in \exp X\mid
F\in\bigcap \mathrm{H}\text{ for some }\mathrm{H}\in \mathbf{F}\},
\mathbf{F}\in G^2X.
$$

In the \emph{capacity monad}
$\BBM=(M,\eta_M,\mu_M)$~\cite{ZN:CapFun:08} the components of the
unit and the multiplication are defined as follows:
$$
\eta_M(x)(F)=
\begin{cases}
1,x\in F,\\
0,x\notin F,
\end{cases}
\;x\in X, F\subcl X,
$$
and
$$
\mu_MX(\CCC)(F)=
\sup\{\alpha\in I\mid
\CCC(\{c\in MX\mid c(F)\ge \alpha\})\ge\alpha
\},
\; \CCC\in M^2, F\subcl X.
$$
An internal relation between the inclusion hyperspace monad and the
capacity monad is presented in \cite{ZN:CapFun:08,Nyk:CAPuniq:0X}.

It is well known that the correspondence which sends each Tychonoff
space $X$ to its Stone-\v Cech compactification $\beta X$ is
naturally extended to a functor $\beta:\Tych\to \Comp$. For a
continuous map $f:X\to Y$ of Tychonoff spaces the map $\beta
f:\beta X\to \beta Y$ is the unique continuous extension of $f$. In
fact this functor is
\emph{left adjoint}~\cite{ML:CWM:98} to the inclusion functor $U$
which embeds $\Comp$ into $\Tych$. The collection
$i=(iX)_{X\in\Ob\Tych}$ of natural embeddings of all Tychonoff
spaces into their Stone-\v Cech compactifications is a unique
natural transformation $\uni{\Tych}\to U\beta$ (a \emph{unit of the
adjunction}, cf.~\cite{ML:CWM:98}).

In this paper ``monotonic'' always means ``isotone''.

\section{Inclusion hyperspace functor and monad in the category
of Tychonoff spaces}

First we modify the Vietoris topology on the set $\exp X$ for a
Tychonoff space $X$. Distinct closed sets in $X$ have distinct
closures in $\beta X$, but the map $e_{\exp} X$ which sends each
$F\in\exp X$ to $\Cl_{\beta X}F\in\exp \beta X$ generally is not an
embedding when the Vietoris topology are considered on the both
spaces, although is continuous. It is easy to prove~:
\begin{lem}
Let $X$ be a Tychonoff space. Then the unique topology on $\exp X$,
such that $e_{\exp} X$ is an embedding into $\exp\beta X$ with the
Vietoris topology, is determined by a base which consists of all
sets of the form
\begin{gather*}
\langle
U_1,\dots,U_k
\rangle
=
\{F\in\exp X\mid
F\text{ is completely separated from }
\\
X\setminus(U_1\cup\dots\cup U_k),
F\cap U_i\ne\ES,i=1,\dots,k\},
\end{gather*}
with all $U_i$ open in $X$.
\end{lem}

Observe that our use of the notation $\langle\dots\rangle$ differs
from its traditional meaning~\cite{TZ:CTCHS:99}, but agrees with it
if $X$ is a compactum. Hence this topology coincides with the
Vietoris topology for each compact Hausdorff space $X$, but may be
weaker for noncompact spaces. The topology is not changed when we
take a less base which consists only of $\langle
U_1,\dots,U_k\rangle$ for $U_i\subop X$ such that $U_2\cup\dots\cup
U_k$ is completely separated from $X\setminus U_1$. We can also
equivalently determine our topology with a subbase which consists
of the sets
$$
\langle U\rangle=
\{F\in\exp X\mid F\text{ is completely separated from }X\setminus U\}
$$
and
$$
\langle X,U\rangle=
\{F\in\exp X\mid F\cap U\ne\ES\}
$$
with $U$ running over all open subsets of $X$.

Observe that the sets of the second type form a subbase of the
lower topology $\tau_l$ on $\exp X$, while a subbase which consists
of the sets of the first form determines a topology that is equal
or weaker than the upper topology $\tau_u$ on $\exp X$. We call it
an~\emph{upper separation topology} (not only for Tychonoff spaces)
and denote by $\tau_{us}$. Thus the topology introduced in the
latter lemma is a lowest upper bound of $\tau_l$ and $\tau_{us}$.
>From now on we
\emph{always} consider $\exp X$ with this topology, if otherwise is
not specified. We also denote by $\expl X$, $\expu X$ and $\expus
X$ the set $\exp X$ with the respective topologies.

If $f:X\to Y$ is a continuous map of Tychonoff spaces, then we
define the map $\exp f:\exp X\to \exp Y$ by the formula $\exp
f(F)=\Cl f(F)$. The equality $e_{\exp} Y\circ \exp f=\exp \beta
f\circ e_{\exp}X$ implies that $\exp f$ is continuous, and we
obtain an extension of the functor $\exp$ in $\Comp$ to $\Tych$.
Unfortunately, the extended functor $\exp$ does not preserve
embeddings.

Now we consider how to define ``valid'' inclusion hyperspaces in
Tychonoff spaces.
\begin{lem}
Let a family $\CCF$ of non-empty closed sets of a Tychonoff space
$X$ is such that $A\subset B\subcl X$, $A\in\CCF$ imply $B\in\CCF$.
Then the following properties are equivalent~:

(a) $\CCF$ is a compact set in $\expl X$;

(b) for each monotonically decreasing net $(F_\alpha)$ of elements
of $\CCF$ the intersection $\bigcap\limits_\alpha F_\alpha$ also is
in $\CCF$.

\smallskip
\noindent Each such $\CCF$ is closed in $\expus X$, hence in $\exp X$.
If $X$ is compact, then these conditions are also equivalent to~:

(c) $\CCF$ is an inclusion hyperspace.
\end{lem}

\begin{proof}
Assume (a), and let $(F_\alpha)$ be a monotonically decreasing net
of elements of $\CCF$. If $\bigcap\limits_\alpha F_\alpha\notin
\CCF$, then the collection $\{\langle X,X\setminus F_\alpha\rangle\}$
is an open cover of $\CCF$ that does not contain a finite subcover,
which contradicts the compactness of $\CCF$ in the lower topology.
Thus (a) implies (b).

Let (b) hold, and we have a cover of $\CCF$ by subbase elements
$\langle X,U_\alpha\rangle$, $\alpha\in\CCA$. If there is no finite
subcover, then $\CCF$ contains all sets of the form $X\setminus
(U_{\alpha_1})
\cup\dots\cup U_{\alpha_k}$, $\alpha_1,\dots,\alpha_k\in\CCA$.
These sets form a filtered family, which may be considered as a
monotonically decreasing net of elements of $\CCF$. Hence, by the
assumption, $\CCF$ contains their non-empty intersection
$B=X\setminus\bigcup\limits_{\alpha\in\CCA}U_\alpha$ that does not
intersect any of $U_\alpha$. This contradiction shows that each
open cover of $\CCF$ by subbase elements contains a finite
subcover, and by Alexander Lemma $\CCF$ is compact, i.e. (a) is
valid.

Let $\CCF$ satisfy (b), and let $C$ be a point of closure of $\CCF$
in $\expus X$. Then for each neighborhood $U\supset C$ there is
$F\in\CCF$ such that $F$ is completely separated from $X\setminus
U$, therefore $\Cl U\in\CCF$. The set $\CCU$ of all closures $\Cl
U$, with $U$ a neighborhood of $C$, is filtered. Therefore
$\bigcap\CCU=C\in\CCF$, hence $\CCF$ is closed in $\expus X$. If
$X$ is a compactum, then $\CCF$ satisfies the definition of
inclusion hyperspace, i.e. (c) is true.

It is also obvious that an inclusion hyperspace on a compactum
satisfies (b).
\end{proof}

Therefore we call a collection $\CCF$ of non-empty closed sets of a
Tychonoff space $X$ a~\emph{compact inclusion hyperspace} in $X$ if
$A\subset B\subcl X$, $A\in\CCF$ imply $B\in\CCF$, and $\CCF$ is
compact in the lower topology on $\exp X$. Note that the lower
topology is non-Hausdorff for non-degenerate $X$. The set of all
compact inclusion hyperspaces in $X$ will be denoted by $\check
GX$.

Let $G^*X$ be the set of all inclusion hyperspaces $\CCG$ in $\beta
X$ with the property~: if $A,B\subcl\beta X$, $A\cap X=B\cap X$,
then $A\in\CCG\iff B\in\CCG$. Observe that each such $\CCG$ does
not contain subsets of $\beta X\setminus X$.

The latter lemma implies~:
\begin{stat}
A collection $\CCF\subset \exp X$ is a compact inclusion hyperspace
if and only if it is equal to $\{G\cap X\mid G\in\CCG\}$ for a
unique $\CCG\in G^*X$ .
\end{stat}
We denote the map $\check GX\to G\beta X$ which sends each
$\CCF\in\check GX$ to the respective $\CCG$ by $e_GX$. It is easy
to see that $e_GX(\CCF)$ is equal to $\{G\in\exp \beta X\mid G\cap
X\in\CCF\}$.

We define a Tychonoff topology on $\check GX$ by the requirement
that $e_GX$ is an embedding into $G\beta X$. An obvious inclusion
$G\beta f(G^*X)\subset G^*Y$ for a continuous map $f:X\to Y$ allows
to define a continuous map $\check Gf:\check GX\to\check GY$ as a
restriction of the map $G\beta f$, i.e. by the equality $G\beta
f\circ e_GX=e_GY\circ\check Gf$. Of course, $\check
Gf(\CCF)=\{G\subcl Y\mid G\supset f(F)\text{ for some } F\in\CCF\}$
for $\CCF\in\check GX$. A functor $\check G$ in the category of
Tychonoff spaces is obtained. Its definition implies that
$e_G=(e_GX)_{X\in\Ob\Tych}$ is a~natural transformation $\check
G\to UG\beta$, with all components being embeddings, therefore
$\check G$ is a subfunctor of $UG\beta$. Note also that
$e_GX=\check GiX$ for all Tychonoff spaces $X$.

Due to the form of the standard subbase of $G\beta X$, we obtain~:
\begin{stat}
The topology on $\check GX$ can be determined by a subbase which
consists of all sets of the form
\begin{gather*}
U^+=
\{\CCF\in\check GX\mid \text{there is }F\in\CCF,
F\text{ is completely separated from }X\setminus U\},
\\
U^-=
\{\CCF\in\check GX\mid F\cap U\ne\ES
\text{ for all }F\in\CCF\},
\end{gather*}
with $U$ open in $X$.
\end{stat}
Observe that this interpretation of $U^+,U^-$ for Tychonoff spaces
agrees with the standard one for compact Hausdorff spaces.

As it was said before, the functor $\exp:\Tych\to\Tych$ does not
preserve embeddings, thus we cannot regard $\exp\exp X$ as a
subspace of $\exp\exp\beta X$, although $\exp X$ is a subspace of
$\exp\beta X$. We can only say that image under $\exp$ of the
embedding $\exp X\to\exp\beta X$ is continuous. Therefore a
straightforward attempt to embed $\check GX$ into $\exp^2X$ fails,
while $\check GX$ is embedded into $\exp^2\beta X$.

Now we will show that the topology on $\check GX$ is the weak
topology with respect to a collection of maps into the unit
interval.

\begin{lem}\label{lm.sup-cont}
Let a map $\phi: X\to I$ be continuous. Then the map $\psi:\exp
X\to I$ which sends each non-empty closed subset $F\subset X$ to
$\sup\limits_{x\in F}\phi(x)$ (or $\inf\limits_{x\in F}\phi(x)$) is
continuous.
\end{lem}

\begin{proof}
We prove for $\sup$, the other case is analogous. Let
$\sup\limits_{x\in F}\phi(x)=\beta<\alpha$, $\alpha,\beta\in I$.
The set $U=\phi^{-1}([0;\frac{\alpha+\beta}2))$ is open, and $F$ is
completely separated from $X\setminus U$, hence $F\in \langle
U\rangle$. If $G\in\exp X$, $G\in \langle U\rangle$, then
$\sup\limits_{x\in F}\phi(x)\le \frac{\alpha+\beta}2<\alpha$ as
well, and the preimage of the set $[0;\alpha)$ under the map $\psi$ is open.

Now let $\sup\limits_{x\in F}\phi(x)=\beta>\alpha$,
$\alpha,\beta\in I$. There exists a point $x\in F$ such that
$\phi(x)>\frac{\alpha+\beta}2$, hence $F$ intersects the open set
$U=\phi^{-1}((\frac{\alpha+\beta}2;1])$. Then $\langle
X,U\rangle\ni F$, and $G\in\exp X$, $G\in\langle X,U\rangle$
implies $\sup\limits_{x\in G}\phi(x)\ge
\frac{\alpha+\beta}2>\alpha$. Therefore the preimage $\psi^{-1}(\alpha;1]$
is open as well, which implies the continuity of $\psi$.
\end{proof}

\begin{lem}
Let a function $\psi:\exp X\to I$ be continuous and monotonic. Then
$\phi$ attains its minimal value on each compact inclusion
hyperspace $\CCF\in\check GX$.
\end{lem}

\begin{proof}
If $\psi$ is continuous and monotonic, then it is lower
semicontinuous with respect to the lower topology. Then the image
of the compact set $\CCF$ under $\psi$ is compact in the topology
$\{I\cap (a,+\infty)\mid a\in\BBR\}$ on $I$, therefore $\psi(\CCF)$
contains a least element.
\end{proof}

\begin{stat}
The topology on $\check GX$ is the weakest among topologies such
that for each continuous function $\phi:X\to I$ the map $m_\phi$
which sends each $\CCF\in \check GX$ to $\min\{\sup\limits_F
\phi\mid F\in\CCF\}$ is continuous. If $\psi:\exp X\to I$ is a
continuous monotonic map, then the map which sends each $\CCF\in
\check GX$ to $\min\{\psi(F)\mid F\in\CCF\}$ is
continuous w.r.t. this topology.
\end{stat}

\begin{proof}
Let $\psi:\exp X\to I$ be a continuous monotonic map, and
$\min\{\psi(F)\mid F\in\CCF\}<\alpha$, then there is $F\in\CCF$
such that $\psi(F)<\alpha$. Due to continuity there is a
neighborhood $\langle U_1,\dots,U_k\rangle\ni F$ such that
$\psi(G)<\alpha$ for all $G\in\langle U_1,\dots,U_k\rangle$. For
$\phi$ is monotonic, the inequality $\psi(G)<\alpha$ is valid for
all $G\in \langle U_1\cup\dots\cup U_k\rangle$. Therefore
$\min\{\psi(G)\mid G\in\CCG\}<\alpha$ for all $\CCG\in
(U_1\cup\dots\cup U_k)^+$, and the latter open set contains $\CCF$.

If $\min\{\psi(F)\mid F\in\CCF\}>\alpha$, then $\psi(F)>\alpha$ for
all $F\in\CCF$. The function $\psi$ is continuous, hence each
$F\in\CCF$ is in a basic neighborhood $\langle
U_0,U_1,\dots,U_k\rangle$ in $\exp X$ such that for all $G$ in this
neighborhood the inequality $\psi(G)>\alpha$ holds. We can assume
that $U_1\cup U_2\cup\dots\cup U_k$ is completely separated from
$X\setminus U_0$, then $\psi(G)>\alpha$ also for all $G\in \langle
X,U_1,U_2,\dots,U_k\rangle$. The latter set is an open neighborhood
of $F$ in the lower topology. The set $\CCF$ is compact in $\expl
X$, therefore we can choose a finite subcover $\langle
U^1_1,\dots,U^1_{k_1}$, \dots, $\langle U^n_1,\dots,U^n_{k_n}$ of
$\CCF$ such that $G\in \langle U^l_1,\dots,U^l_{k_l}\rangle$, $1\le
l\le n$, implies $\psi(G)>\alpha$. Then $\CCF$ is in an open
neighborhood
$$
\CCU=\bigcap\{(U^1_{j_1}\cup
U^2_{j_2}\cup\dots\cup U^n_{j_n})^-
\mid
1\le j_1\le k_1,
2\le j_2\le k_2,
\dots
n\le j_n\le k_n\}.
$$
Each element $G$ of any compact inclusion hyperspace $\CCG\in\CCU$
intersects all $U^l_1,\dots,U^l_{k_l}$ for at least one
$l\in\{1,\dots,n\}$, therefore $\min\{\psi(G)\mid
G\in\CCG\}>\alpha$ for all $\CCG\in \CCU$. Thus $\min\{\psi(F)\mid
F\in\CCF\}$ is continuous w.r.t. $\CCF\in\check GX$.

Due to Lemma~\ref{lm.sup-cont} it implies that the map $m:\check
GX\to I^{C(X,I)}$, $m(\CCF)=(m_\phi(\CCF))_{\phi\in C(X,I)}$ for
$\CCF\in\check GX$, is continuous.

Now let $\CCF\in U^+$ for $U\subop X$, i.e. there is $F\in\CCF$ and
a continuous function $\phi:X\to I$ such that $\phi|_F\equiv 0$,
$\phi|_{X\setminus U}=1$. Then $m_\phi(\CCF)<1/2$, and for any
$\CCG\in\check GX$ the inequality $m_\phi(\CCG)<1/2$ implies
$\CCG\in U^+$.

If $\CCF\in U^-$, $U\subop X$, then due to the compactness of
$\CCF$ we can choose $V\subop X$ such that $\CCF\in V^-$, and there
is a continuous map $\phi:X\to I$ such that $\phi|_V=1$,
$\phi|_{X\setminus U}=0$. Then $m_\phi(\CCF)=1>1/2$, and for each
$\CCG\in\check GX$ the inequality $m_\phi(\CCG)>1/2$ implies
$\CCG\in U^-$. Therefore the inverse to $m$ is continuous on
$m(\check GX)$, thus the map $m:\check GX\to I^{C(X,I)}$ is an
embedding, which completes the proof.
\end{proof}

\begin{rem}
It is obvious that the topology on $\check GX$ can be equivalently
defined as the weak topology w.r.t. the collection of maps
$m^\phi:\check GX\to I$, $m^\phi(\CCF)=\max\{\inf\limits_F
\phi\mid F\in\CCF\}$, for all $\phi\in C(X,I)$.
\end{rem}

Further we will need the subspace
$$
\hat GX=\{\CCF\in\check GX\mid \text{for all }F\in\CCF
\text{ there is a compactum }K\subset F,K\in \CCF
\}\subset \check GX.
$$
It is easy to see that its image under $e_GX:\check GX\hra G\beta X$
is the set
$$
G_*X=\{\CCG\in G^*X\mid \text{for all }G\in\CCG
\text{ there is a compactum }K\subset G\cap X,K\in \CCG
\},
$$
and $\check Gf(\hat GX)\subset \hat GY$ for each continuous map $f:X\to Y$ of
Tychonoff spaces. Thus we obtain a subfunctor $\hat G$
of the functor $\check G:\Tych\to \Tych$.

\begin{lem}
Let $X$ be a Tychonoff space. Then $\mu_G\beta X\circ
\check Ge_GX(\check G^2X)\subset e_GX(\check GX)$.
\end{lem}

The composition in the above inclusion is legal because $\check
G\beta X=G\beta X$.

\begin{proof}
Let $\mathbf{F}\in \check G^2X$, $\CCF=\mu_G\beta X\circ
\check Ge_GX(\mathbf{F})$, and $F,G\subcl\beta X$ are
such that $F\cap X=G\cap X$. Assume $F\in\CCF$, then there is
$\mathrm{H}\in \mathbf{F}$ such that $F\in\CCG$ for all
$\CCG\in\Cl_{G\beta X}e_GX(\mathrm{H})$, therefore for all $\CCG\in
e_GX(\mathrm{H})$. It is equivalent to $F\cap X\in \CCH$ for all
$\CCH\in \mathrm{H}\subset \check GX$, which in particular implies
that $F\cap X\ne \ES$. By the assumption, $G\cap X\in \CCH$ for all
$\CCH\in\mathrm{H}$ as well, hence $G\in\CCG$ for all $\CCG\in
e_GX(\mathrm{H})$. The set of all $\CCH\in\check GX$ such that
$\CCH\ni A$ is closed for any $A\in\exp X$, thus $G\in\CCG$ for all
$\CCG\in \Cl_{G\beta X}e_GX(\mathrm{H})$. We infer that $G\in\CCF$,
and $\CCF\in\check GX$.
\end{proof}

For $e_GX$ is an embedding, we define $\check \mu_GX$ as a map
$\check G^2X\to\check GX$ such that $e_GX\circ
\check\mu_GX=\mu_G\beta X\circ \check G e_GX$. This map is unique
and continuous. Following the latter proof, we can see that
$$
\check \mu_G(\mathbf{F})=
\{F\in \exp X\mid
F\in\bigcap \mathrm{H}\text{ for some }\mathrm{H}\in \mathbf{F}\},
\mathbf{F}\in \check G^2X,
$$
i.e. the formula is the same as in $\Comp$.

For the inclusion $\eta_G\beta X\circ iX(X)\subset e_GX(\check GX)$
is also true, there is a~unique map $\check
\eta_GX:X\to \check GX$ such that $e_GX\circ \eta_GX=\eta_G\beta X\circ
iX$, namely $\check\eta_GX(x)=\{F\in\exp X\mid F\ni x\}$ for each
$x\in X$, and this map is continuous. It is straightforward to
prove that the collections
$\check\eta_G=(\check\eta_GX)_{X\in\Ob\Tych}$ and
$\check\mu_G=(\check\mu_GX)_{X\in\Ob\Tych}$ are natural
transformations resp.\ $\uni{\Tych}\to\check G$ and $\check G^2
\to\check G$.
\begin{theo}\label{st.gcheck-mon}
The triple $\check\BBG=(\check G,\check \eta_G,\check \mu_G)$ is a
monad in $\Tych$.
\end{theo}

\begin{proof}
Let $X$ be a Tychonoff space and $iX$ its embedding into $\beta X$.
Then~:
\begin{gather*}
e_GX\circ \check \mu X\circ \check \eta \check GX=
\mu\beta X\circ\check Ge_GX\circ\check\eta \check GX=
\\
\mu\beta X\circ\eta G\beta X\circ e_GX=
\uni{G\beta X}\circ e_GX=e_GX,
\end{gather*}
thus $\check\mu_G X\circ \check G\check\eta X=\check\mu_G X\circ
\check\eta_G\check GX=\uni{\check GX}$, similarly we obtain the
equalities $\check \mu_G X\circ \check G\check\eta_G X=\uni{\check
GX}$ and $\check\mu_G X\circ \check G\check \mu_G X=\check\mu_G
X\circ\check\mu_G\check GX$.
\end{proof}

For $\check GX$, $\check \eta_GX$, $\check\mu_GX$ coincide with
$GX$, $\eta_GX$, $\mu_GX$ for any compactum $X$, the monad $\check
\BBG$ is an extension of the monad $\BBG$ in $\Comp$ to $\Tych$.

\section{Functional representation of the capacity monad in the category of compacta}

In the sequel $X$ is a compactum, $c$ is a capacity on $X$ and
$\phi:X\to \BBR$ is a continuous function. We define the
\emph{Sugeno integral} of $\phi$ with respect to~$c$ by the
formula~\cite{Nyk:SugInTop:0X}~:
$$
\jint_X \phi(x)\land dc(x)=
\sup\{c(\{x\in X\mid \phi(x)\ge \alpha\})\land \alpha
\mid \alpha\in I\}.
$$

The following theorem was recently obtained (in an equivalent form)
by Radul~\cite{Rad:FuncRepr:09} under more restrictive conditions,
namely restrictions of normalizedness and non-expandability were
also imposed. Therefore for the readers convenience we provide a
formulation and a short proof of a version more suitable for our
needs.

\begin{theo}\label{th.char-sugeno-i}
Let $X$ be a compactum, $c$ a capacity on $X$. Then the functional
$i:C(X,I)\to I$, $i(\phi)=\jint_X
\phi(x)\land dc(x)$ for $\phi\in C(X,I)$, has the following properties~:
\begin{enumerate}
\item for all $\phi,\psi\in C(X,I)$ the inequality $\phi\le\psi$ (i.e.\
$\phi(x)\le \psi(x)$ for all $x\in X$) implies $i(\phi)\le i(\psi)$
($i$ is \emph{monotonic});
\item $i$
satisfies the equalities $i(\alpha\land\phi)=\alpha\land i(\phi)$,
$i(\alpha\lor\phi)=\alpha\lor i(\phi)$ for any $\alpha\in I$,
$\phi\in C(X,I)$.
\end{enumerate}

Conversely, any functional $i:C(X,I)\to I$ satisfying (1),(2) has
the form $i(\phi)=\jint_X\phi(x)\land dc(x)$ for a uniquely
determined capacity $c\in MX$.
\end{theo}

In the two following lemmata $i:C(X,I)\to I$ is a functional that
satisfies (1),(2).

\begin{lem}
If $\alpha\in I$ and continuous functions $\phi,\psi:X\to I$ are
such that $\{x\in X\mid\phi(x)\ge\alpha\}\subset
\{x\in X\mid\psi(x)\ge\alpha\}$ and $i(\phi)\ge\alpha$, then
$i(\psi)\ge\alpha$.
\end{lem}

\begin{proof}
For $\alpha=0$ the statement is trivial. Otherwise assume
$i(\phi)\ge\alpha$. Let $0\le\beta<\alpha$. For $\phi,\psi$ are
continuous, there is $\gamma\in(\beta;\alpha)$ such that the closed
sets $F=\psi^{-1}([0;\beta])$ and $G=\phi^{-1}([\gamma,1])$ have an
empty intersection. Then, by Brouwer-Tietze-Urysohn Theorem, there
is a continuous function $\theta:X\to [\beta;\gamma]$ such that
$\theta|_F\equiv \beta$, $\theta|_G\equiv \gamma$. Then we define a
function $f:X\to I$ as follows~:
$$
f(x)=
\begin{cases}
\psi(x), x\in F,\\
\theta(x),x\notin F\cup G,\\
\phi(x),x\in G.
\end{cases}
$$
Then $\gamma\lor f=\gamma\lor \phi$, thus
$$
\gamma\lor i(f)=i(\gamma\lor f)= i(\gamma\lor \phi)=\gamma\lor
i(\phi)=\alpha,
$$
and $i(f)=\alpha$. Taking into account $\beta\land
f=\beta\land\psi$, we obtain
$$
\beta=\beta\land i(f)=i(\beta\land f)=i(\beta\land
\psi)=\beta\land i(\psi),
$$
thus $i(\psi)\ge \beta$ for all $\beta<\alpha$. It implies
$i(\psi)\ge\alpha$.
\end{proof}

Obviously if $\{x\in X\mid\phi(x)\ge\alpha\}= \{x\in
X\mid\psi(x)\ge\alpha\}$, then $i(\phi)\ge\alpha$ if and only if
$i(\psi)\ge\alpha$.

\begin{lem}
For each closed set $F\subset X$ and $\beta\in I$ the equality
$$
\inf\{i(\phi)\mid \phi\ge\alpha\land \chi_F\}=
\alpha\land\inf\{i(\psi)\mid \phi\ge \chi_F\}
$$
is valid.
\end{lem}

\begin{proof}
It is sufficient to observe that for all $0\le\beta<\alpha$ the
sets $\{\beta\land \phi\mid \phi\ge\alpha\land \chi_F\}$ and
$\{\beta\land\psi\mid \psi\ge \chi_F\}$ coincide, therefore by the
previous lemma~:
$$
\beta\land \inf\{i(\phi)\mid \phi\ge\alpha\land \chi_F\}=
\beta\land \inf\{i(\psi)\mid \psi\ge \chi_F\}=
\beta\land \alpha\land \inf\{i(\psi)\mid \psi\ge \chi_F\}.
$$
For the both expressions $\inf\{i(\phi)\mid \phi\ge\alpha\land
\chi_F\}$ and $\alpha\land \inf\{i(\psi)\mid \psi\ge \chi_F\}$ do
not exceed $\alpha$, they are equal.
\end{proof}

\begin{proof}[Proof of the theorem]
It is obvious that Sugeno integral w.r.t. a~capacity satisfies
(1),(2). If $i$ is Sugeno integral w.r.t. some capacity $c$, then
the equality $c(F)=\inf\{i(\psi)\mid \psi\ge \chi_F\}$ must hold
for all $F\subcl X$. To prove the converse, we assume that
$i:C(X,I)\to I$ satisfies (1),(2) and use the latter formula to
define a set function~$c$. It is obvious that the first two
conditions of the definition of capacity hold for $c$. To show
upper semicontinuity, assume that $c(F)<\alpha$ for some $F\subcl
X$, $\alpha\in I$. Then there is a continuous function $\phi:X\to
I$ such that $\phi\ge\chi_F$, $i(\phi)<\alpha$. Let
$i(\phi)<\beta<\alpha$, then
$$
i(\phi)=\beta\land i(\phi)=i(\beta\land\phi)\ge
\beta\land c(\{x\in X\mid \phi(x)\ge\beta\}),
$$
which implies $c(\{x\in X\mid \phi(x)\ge\beta\})<\beta<\alpha$. The
set $U=\{x\in X\mid \phi(x)>\beta\}$ is an open neighborhood of $F$
such that $c(G)<\alpha$ for all $G\subcl X$, $G\subset U$. Thus $c$
is upper semicontinuous and therefore it is a capacity.

The two previous lemmata imply that for any $\phi\in C(X,I)$ we
have
\begin{gather*}
i(\phi)=\sup\{\alpha\in I\mid i(\phi)\ge\alpha\}=
\sup\{\alpha\in I\mid c(\{x\in X\mid
\phi(x)\ge\alpha\})\ge\alpha\}=\\
\sup\{\alpha\land c(\{x\in X\mid \phi(x)\ge\alpha\})\mid\alpha\in I\}=
\jint\limits_X \phi(x)\land dc(x).
\end{gather*}
\end{proof}

\begin{lem}
Let $\phi:X\to I$ be a continuous function. Then the map
$\delta_\phi:MX\to I$ which sends each capacity $c$ to
$\jint\limits_X
\phi(x)\land dc(x)$ is continuous.
\end{lem}

\begin{proof}
Observe that
$$
\delta_\phi{}^{-1}(([0;\alpha))=
O_-(\phi^{-1}([0;\alpha]),\alpha),
\quad
\delta_\phi{}^{-1}(((\alpha;1])=
O_+(\phi^{-1}((\alpha;1]),\alpha)
$$
for all $\alpha\in I$.
\end{proof}

\begin{cons}
The map $X\to I^{C(X,I)}$ which sends each capacity $c$ on $X$ to
$(\delta_\phi(c))_{\phi\in C(X,I)}$ is an embedding.
\end{cons}

Recall that its image consists of all monotonic functionals from
$C(X,I)$ to $I$ which satisfy (1),(2). Therefore from now on we
identify each capacity and the respective functional. By the latter
statement the topology on $MX$ can be equivalently defined as
weak${}^*$ topology using Sugeno integral instead of Choquet
integral. We also write $c(\phi)$ for $\jint\limits_X
\phi(x)\land dc(x)$.

The following observation is a trivial ``continuous'' version of
Theorem~6.5~\cite{Nyk:SugInTop:0X}.

\begin{stat}
Let $C\in MX$ and $\phi\in C(X,I)$. Then $\mu_M
X(C)(\phi)=C(\delta_\phi)$.
\end{stat}

\begin{proof}
Indeed, the both sides are greater or equal than $\alpha\in I$ if
and only if $C\{c\in MX\mid c(\phi)\ge\alpha\}\ge\alpha$.
\end{proof}

It is also easy to see that $\eta_M X(x)(\phi)=\phi(x)$ for all
$x\in X$, $\phi\in C(X,I)$. Thus we have obtained a description of
the capacity monad $\BBM$ in terms of functionals which is a
complete analogue of the description of the probability monad
$\BBP$~\cite{Fed:FuncProbMeasTopCat:98,TZ:CTCHS:99}. Now we can
easily reprove the continuity of $\eta_M X$ and $\mu_M X$, as well
as the fact that $\BBM=(M,\eta_M,\mu_M)$ is a monad.

\section{Extensions of the capacity functor and the capacity monad
to the category of Tychonoff spaces}

We will extend the definition of capacity to Tychonoff spaces. A
function $c:\exp X\cup\{\ES\}\to I$ is called a~\emph{regular
capacity} on a Tychonoff space $X$ if it is monotonic, satisfies
$c(\ES)=0$, $c(X)=1$ and the following property of
\emph{upper semicontinuity} or \emph{outer regularity}~:
if $F\subcl X$ and $c'(F)<\alpha$, $\alpha\in I$, then there is an
open set $U\supset F$ in $X$ such that $F$ and $X\setminus U$ are
completely separated, and $c'(G)<\alpha$ for all $G\subset U$,
$G\subcl X$.


This definition implies that each closed set $F$ is contained in
some zero-set $Z$ such that $c(F)=c(Z)$.

Each capacity $c$ on any compact space $Y$ satisfies also the
property which is called $\tau$\emph{-smoothness} for additive
measures and have two slightly different
formulations~\cite{Ban:TopSpProbMeas-I:95,Var:MeasOnTopSp:61}.
Below we show that they are equivalent for Tychonoff spaces.
\begin{lem}
Let $X$ be a Tychonoff space and $m:\exp X\cup\{\ES\}\to I$
a~monotonic function. Then the two following statements are
equivalent~:

(a) for each monotonically decreasing net $(F_\alpha)$ of closed
sets in $X$ and a closed set $G\subset X$, such that
$\bigcap\limits_{\alpha}F_\alpha\subset G$, the inequality
$\inf\limits_\alpha c(F_\alpha)\le c(G)$ is valid;

(b) for each monotonically decreasing net $(Z_\alpha)$ of zero-sets
in $X$ and a closed set $G\subset X$, such that
$\bigcap\limits_{\alpha}Z_\alpha\subset G$, the inequality
$\inf\limits_\alpha c(Z_\alpha)\le c(G)$ is valid.
\end{lem}

\begin{proof}
It is obvious that (a) implies (b). Let (b) hold, and let a net
$(F_\alpha)$ and a set $G$ satisfy the conditions of (a). We denote
the set of all pairs $(F_\alpha,a)$ such that $a\in X\setminus
F_\alpha$ by $A$, and let $\Gamma$ be the set of all non-empty
finite subsets of $A$. The space $X$ is Tychonoff, hence for each
pair $(F_\alpha,a)\in A$ there is a zero-set $Z_{\alpha,a}\supset
F_\alpha$ such that $Z_{\alpha,a}\not\ni a$. For
$\gamma=\{(\alpha_1,a_1),\dots,(\alpha_k,a_k)\}\in \Gamma$ we put
$Z_\gamma=Z_{\alpha_1,a_1}\cap\dots\cap Z_{\alpha_k,a_k}$. If
$\Gamma$ is ordered by inclusion, then $(Z_\gamma)_{\gamma\in
\Gamma}$ is a monotonically decreasing net such that
$\bigcap\limits_{\gamma\in\Gamma}Z_\gamma=
\bigcap\limits_{\alpha}F_\alpha\subset G$, thus
$\inf\limits_\alpha c(F_\alpha)\le \inf
\limits_{\gamma\in\Gamma}Z_\gamma\le c(G)$, and (a) is valid.
\end{proof}

We call a function $c:\exp X\to I$ a~$\tau$\emph{-smooth capacity}
if it is monotonic, satisfies $c(\ES)=0$, $c(X)=1$ and any of the
two given above equivalent properties of $\tau$-smoothness. It is
obvious that each $\tau$-smooth capacity is a regular capacity, but
the converse is false. E.g. the function $c:\exp
\BBN\cup\{\ES\}\to I$ which is defined by the formulae $c(\ES)=0$,
$c(F)=1$ as $F\subset \BBN$, $F\ne \ES$, is a regular capacity that
is not $\tau$-smooth. For compacta the two classes coincide.

>From now all capacities are $\tau$-smooth, if otherwise is not
specified.

Now we show that capacities on a Tychonoff space $X$ can be
naturally identified with capacities with a certain property on the
Stone-\v Cech compactification $\beta X$.

\begin{lem}
Let $c$ be a capacity on $\beta X$. Then the following statements
are equivalent~:

(1) for each closed sets $F,G\subset \beta X$ such that $F\cap
X\subset G$, the inequality $c(F)\le c(G)$ is valid;

(2) for each monotonically decreasing net $(\phi_\gamma)$ of
continuous functions $\beta X\to I$ and a continuous function
$\psi:\beta X\to I$ such that $\inf\limits_\gamma
\phi_\gamma(x)\le \psi(x)$ for all $x\in X$, the inequality
$\inf\limits_\gamma c(\phi_\gamma)\le c(\psi)$ is valid.
\end{lem}

\begin{proof}
(1)$\implies$(2). Let $c(\psi)<\alpha$, $\alpha\in I$, then
$c(Z_0)<\alpha$ for the closed set $Z_0=\{x\in \beta X\mid
\psi(x)\ge\alpha\}$. The intersection $Z$ of the closed sets
$Z_\gamma=\{x\in \beta X\mid \phi_\gamma(x)\ge\alpha\}$ satisfies
the inclusion $Z\cap X\subset Z_0$, hence by (1): $c(Z)\le c(Z_0)$.
Due to $\tau$-smoothness of $c$ we obtain $\inf\limits_\gamma
c(Z_\gamma)\le c(Z)$. Therefore there exists an index $\gamma$ such
that $c(\{x\in\beta X\mid\phi_\gamma(x)\ge\alpha\})<\alpha$, thus
$c(\phi_\gamma)<\alpha$, and $\inf\limits_\gamma
c(\phi_\gamma)<\alpha$, which implies the required inequality.

(2)$\implies$(1). Let a continuous function $\psi:\beta X\to I$ be
such that $\psi|_G=1$. Denote the set of all continuous functions
$\phi:\beta X\to I$ such that $\phi|_F\equiv 1$ by $\CCF$. We
consider the order on $\CCF$ which is reverse to natural:
$\phi\prec\phi'$ if $\phi\ge\phi'$, then the collection $\CCF$ can
be regarded as a monotonically decreasing net such that
$(\phi(x))_{\phi\in\CCF}$ converges to $1$ for all $x\in X\cap G$,
and to $0$ for all $x\in X\setminus G$. Therefore
$\inf\limits_{\phi\in\CCF}\phi(x)\le
\psi(x)$ for all $x\in X$, hence, by the assumption:
$\inf\limits_{\phi\in\CCF}c(\phi)\le c(\psi)$. Thus
\begin{multline*}
\inf\{c(\phi)\mid \phi:\beta X\to I\text{ is continuous},\phi|_F\equiv 1\}
\le
\\
\inf\{c(\psi)\mid \psi:\beta X\to I\text{ is continuous},\psi|_G\equiv
1\},
\end{multline*}
i.e. $c(F)\le c(G)$.
\end{proof}

We define the set of all $c\in M\beta X$ that satisfy (1)$\iff$(2)
by $M^*X$.

Condition (1) implies that, if closed sets $F,G\subset \beta X$ are
such that $F\cap X=G\cap X$, then $c(F)=c(G)$. Therefore we can
define a set function $\check c:\exp X\cup\{\ES\}\to I$ as
follows~: if $A\subcl X$, then $\check c(A)=c(F)$ for any set
$F\subcl\beta X$ such that $F\cap X= A$. Obviously $\check
c(A)=\inf\{c(\psi)\mid \psi\in C(\beta X,I), \psi\ge\chi_A\}$.

The following observation, although almost obvious, is a crucial
point in our exposition.
\begin{stat}
A set function $c':\exp X\cup\{\ES\}\to I$ is equal to $\check c$
for some $c\in M^*X$ if and only if $c'$ is a $\tau$-smooth
capacity on $X$.
\end{stat}

Therefore we define the set of all capacities on $X$ by $\check MX$
and identify it with the subset $M^* X\subset M\beta X$. We obtain
an injective map $e_MX:\check MX\to M\beta X$, and from now on we
assume that a topology on $\check MX$ is such that $e_MX$ is an
embedding. Thus $\check MX$ for a Tychonoff $X$ is Tychonoff as
well.

If $c$ is a capacity on $X$ and $\phi:X\to I$ is a continuous
function, we define the Sugeno integral of $\phi$ w.r.t. $c$ by the
usual formula~:
$$
c(\phi)=\jint\limits_X \phi(x)\land dc(x)=
\sup\{\alpha\land c\{x\in X\mid
\phi(x)\ge\alpha\}\mid
\alpha\in I\}.
$$

For any continuous function $\phi:X\to I$ we denote by $\beta\phi$
its Stone-\v Cech compactification, i.e. its unique continuous
extension to a function $\beta X\to I$.

\begin{stat}
Let $c\in M^*X$ and $\check c$ is defined as above. Then for any
continuous function $\phi:X\to I$ we have $\check
c(\phi)=c(\beta\phi)$.
\end{stat}

\begin{proof}
It is sufficient to observe that
$$
\check c(\{x\in X\mid
\phi(x)\ge\alpha\})=c(\{x\in \beta X\mid
\beta\phi(x)\ge\alpha\}).
$$
\end{proof}

Thus the topology on $\check MX$ can be equivalently defined as the
weak${}^*$-topology using Sugeno integral. It also immediately
implies that the following theorem is valid.

\begin{theo}\label{th.char-sugeno-tych}
Let $X$ be a Tychonoff space, $c$ a~capacity on $X$. Then the
functional $i:C(X,I)\to I$, $i(\phi)=\jint_X
\phi(x)\land dc(x)$ for $\phi\in C(X,I)$, has the following properties~:
\begin{enumerate}
\item for all $\phi,\psi\in C(X,I)$ the inequality $\phi\le\psi$ (i.e.\
$\phi(x)\le \psi(x)$ for all $x\in X$) implies $i(\phi)\le i(\psi)$
($i$ is \emph{monotonic});
\item $i$
satisfies the equalities $i(\alpha\land\phi)=\alpha\land i(\phi)$,
$i(\alpha\lor\phi)=\alpha\lor i(\phi)$ for any $\alpha\in I$,
$\phi\in C(X,I)$;
\item
for each monotonically decreasing net $(\phi_\alpha)$ of continuous
functions $X\to I$ and a continuous function $\psi:X\to I$ such
that $\inf\limits_\alpha
\phi_\alpha(x)\le \psi(x)$ for all $x\in X$, the inequality
$\inf\limits_\alpha i(\phi_\alpha)\le i(\psi)$ is valid.
\end{enumerate}

Conversely, any functional $i:C(X,I)\to I$ satisfying (1)--(3) has
the form $i(\phi)=\jint_X\phi(x)\land dc(x)$ for a uniquely
determined capacity $c\in \check MX$.
\end{theo}

Condition (3) is superfluous for a compact space $X$, but cannot be
omitted for noncompact spaces. E.g. the functional, which sends
each $\phi\in C(\BBR,I)$ to $\sup\phi$, has properties (1),(2), but
fails to satisfy (3).

The following statement is an immediate corollary of an analogous
theorem for the compact case.
\begin{stat}
The topology on $\check MX$ can be equivalently determined by a
subbase which consists of all sets of the form
\begin{gather*}
O_+(U,\alpha)=\{c\in \check MX\mid \text{there is }F\subcl
X,
\\
F\text{ is completely separated from }X\setminus U, c(F)>\alpha\}
\end{gather*}
for all open $U\subset X$, $\alpha\in I$, and of the form
$$
O_-(F,\alpha)=\{c\in\check MX\mid c(F)<\alpha\}
$$
for all closed $F\subset X$, $\alpha\in I$.
\end{stat}

Like the compact case, for a continuous map $f:X\to Y$ of Tychonoff
spaces we define a map $\check Mf:\check MX\to \check MY$ by the
two following equivalent formulae: $\check Mf(c)(F)=c(f^{-1}(F))$,
with $c\in \check MX$, $F\subcl Y$ (if set functions are used), or
$\check Mf(c)(\phi)=c(\phi\circ f)$ for $c\in \check MX$, $\phi\in
C(X,I)$ (if we regard capacities as functionals). The latter
representation implies the continuity of $\check Mf$, and we obtain
a functor $\check M$ in the category $\Tych$ of Tychonoff spaces
that is an extension of the capacity functor $M$ in $\Comp$.

The map $e_MX:\check MX\to M\beta X$ coincides with $\check MiX$,
where $iX$ is the embedding $X\hra \beta X$ (we identify $\check
M\beta X$ and $M\beta X$), and the collection
$e_M=(e_MX)_{X\in\Ob\Tych}$ is a~natural transformation from the
functor $\check M$ to the functor $UM\beta$, with $U:\Comp\to
\Tych$ being the inclusion functor. Observe that $\eta_M
\beta X(X)\subset M^*X=e_MX(\check MX)$, therefore there is a
continuous restriction $\check\eta_M X=\eta_M\beta X|_{X}:X\to
\check MX$ which is a component of a natural transformation
$e_M:\uni{\Tych}\to\check M$. For all $x\in X\in\Ob\Tych$, $F\subcl
X$ the value $\check\eta_GX(x)(F)$ is equal to $1$ if $x\in F$,
otherwise is equal to $0$.

\begin{lem}
Let $X$ be a Tychonoff space. Then $\mu_M\beta X\circ
\check Me_MX(\check M^2X)\subset e_MX(\check MX)$.
\end{lem}

\begin{proof}
Let $C\in \check M^2X$, and $F,G\subcl \beta X$ are such that
$F\cap X\subset G$. Then for all $c\in M^*X$ we have $c(F)\le
c(G)$, thus for each $\alpha\in I$~:
$$
\{c\in \check MX\mid e_MX(c)(F)\ge\alpha\}
\subset
\{c\in \check MX\mid e_MX(c)(G)\ge\alpha\},
$$
hence
\begin{multline*}
\check Me_MX(C)(\{c\in M\beta X\mid c(F)\ge\alpha\}))=
C(e_MX^{-1}(\{c\in M\beta X\mid c(F)\ge\alpha\}))\le
\\
C(e_MX^{-1}(\{c\in M\beta X\mid c(G)\ge\alpha\}))=
\check Me_MX(C)(\{c\in M\beta X\mid c(G)\ge\alpha\})),
\end{multline*}
thus
\begin{gather*}
\mu_M\beta X\circ \check Me_MX(C)(F)=\sup\{\alpha\land
\check Me_MX(C)(\{c\in M\beta X\mid c(F)\ge\alpha\})\le
\\
\sup\{\alpha\land
\check Me_MX(C)(\{c\in M\beta X\mid c(G)\ge\alpha\})=
\mu_M\beta X\circ \check Me_MX(C)(G),
\end{gather*}
which means that $\mu_M\beta X\circ \check Me_MX(C)\in
M^*X=e_MX(\check MX)$.
\end{proof}

For $e_MX:\check MX\to M\beta X$ is an embedding, there is a unique
map $\check\mu_M X:\check M^2X\to \check MX$ such that $\mu_M\beta
X\circ Me_MX=e_MX\circ \check\mu_M X$, and this map is continuous.
It is straightforward to verify that the collection
$\check\mu_M=(\check\mu_M X)_{X\in
\Ob\Tych}$ is a natural transformation $\check M^2\to \check M$, and
$\check \mu_M X$ can be defined directly, without involving Stone-\v
Cech compactifications, by the usual formulae~:
$$
\check\mu_M X(C)(F)=
\sup\{\alpha\land
C(\{c\in \check MX\mid c(F)\ge\alpha\}),\;C\in \check M^2X, F\subcl
X,
$$
or
$$
\check \mu_M X(C)(\phi)=C(\delta_\phi),\phi\in C(X,I),\text{ where }
\delta_\phi(c)=c(\phi)\text{ for all }c\in \check MX.
$$

\begin{theo}
The triple $\check\BBM=(\check M,\check\eta_M,\check\mu_M)$ is a monad
in $\Tych$.
\end{theo}

\emph{Proof} is a complete analogue of the proof of
Proposition~\ref{st.gcheck-mon}.

This monad is an extension of the monad $\BBM=(M,\eta_M,\mu_M)$ in
$\Comp$ in the sense that $\check MX=MX$, $\check\eta_M X=\eta_M X$
and $\check \mu_M X=\mu_M X$ for each compactum~$X$.

\begin{stat}
Let for each compact inclusion hyperspace $\CCF$ on a Tychonoff
space $X$ the set function $i_G^MX(\CCF):\exp X\cup\{\ES\}\to I$ be
defined by the formula
$$
i_G^MX(\CCF)(A)=
\begin{cases}
1, A\in\CCF,\\
0, A\notin \CCF,
\end{cases}
\;A\subcl X.
$$
Then $i_G^KX$ is an embedding $\check GX\hra\check MX$, and the
collection $i_G^K=(i_G^KX)_{X\in\Ob\Tych}$ is a morphism of monads
$\check \BBG\to \check \BBM$.
\end{stat}

Thus the monad $\check\BBG$ is a submonad of the monad $\check
\BBM$.

Now let
$$
M_*X=\{c\in M\beta X\mid
c(A)=
\sup\{c(F)\mid F\subset A\cap X\text{ is compact}\}
\text{ for all }A\subcl
\beta X\}.
$$

It is easy to see that $M_*X\subset M^*X$. As a corollary we obtain

\begin{stat}
A set function $c':\exp X\cup\{\ES\}\to I$ is equal to $\check c$
for some $c\in M_*X$ if and only if $c'$ is a $\tau$-smooth
capacity on $X$ and satisfies the condition $c'(A)=\sup\{c'(F)\mid
F\subset A\text{ is compact}\}$ for all $A\subcl X$ (\emph{inner
compact regularity}).
\end{stat}

If a set function satisfies (1)--(4), we call it a~\emph{Radon
capacity}. The set of all Radon capacities on $X$ is denoted by
$\hat MX$ and regarded as a subspace of~$\check MX$. An obvious
inclusion $M\beta f(M_*X)\subset M_*Y$ for a continuous map $f:X\to
Y$ of Tychonoff spaces implies $\check Mf(\hat MX)\subset
\hat MY$. Therefore we denote the restriction of $\check Mf$
to a mapping $\hat MX\to \hat MY$ by $\hat Mf$ and obtain a
subfunctor $\hat M$ of the functor $\check M$.

\begin{que}
What are necessary and sufficient conditions for a functional
$i:C(X,I)\to I$ to have the form $i(\phi)=\jint_X\phi(x)\land
dc(x)$ for a some capacity $c\in \hat MX$?
\end{que}

Here is a \emph{necessary condition}~: for each monotonically
\emph{increasing} net $(\phi_\alpha)$ of continuous functions $X\to
I$ and a continuous function $\psi:X\to I$ such that
$\sup\limits_\alpha
\phi_\alpha(x)\ge \psi(x)$ for all $x\in X$, the inequality
$\sup\limits_\alpha i(\phi_\alpha)\ge i(\psi)$ is valid.

The problem of existence of a restriction of $\check\mu_M X$ to a
map $\hat M^2X\to \hat MX$ is still unsolved and is connected with
a similar question for inclusion hyperspaces by the following

\begin{stat}\label{caps-to-incl}
Let $X$ be a Tychonoff space. If $\check\mu_M X(\hat M^2X)\subset
\hat MX$, then $\check\mu_G X(\hat G^2X)\subset
\hat GX$.
\end{stat}

\begin{proof} We will consider equivalent inclusions $\mu_M\beta
X(M_*^2X)\subset M_*X$ and $\mu_G\beta X(G_*^2X)\subset G_*X$. The
latter one means that, for each set $A\subcl X$ and compact set
$\CCG\subset G\beta X$ such that each element $F$ of any inclusion
hyperspace $\CCB\in\CCG$ contains a compactum $K\in \CCB$,
$K\subset X$, there is a compact set $H\subset A$,
$H\in\bigcap\CCG$.

Assume that $\mu_G\beta X(G_*^2X)\not\subset G_*X$, then there are
$A\subcl X$ and a compact set $\CCG\subset G_*X$ such that all
inclusion hyperspaces in $\CCG$ contain subsets of $A$, but there
are no compact subsets of $A$ in $\bigcap\CCG$. For each $\CCB\in
\CCG$ let a capacity $c_{\CCB}$ be defined as follows~:
$$
c_{\CCB}(F)=
\begin{cases}
1, F\in\CCB,\\ 0, F\notin \CCB,
\end{cases}
\quad F\subcl \beta X.
$$
It is obvious that $c_\CCB\in M_*X$, and the correspondence
$\CCB\mapsto c_\CCB$ is continuous, thus the set
$\mathrm{B}=\{c_{\CCB}\mid \CCB\in\CCG\}\subset M\beta X$ is
compact. Therefore the capacity $C\in M^2\beta X$, defined as
$$
C(\CCF)=
\begin{cases}
1, \CCF\supset\mathrm{B},\\ 0, \CCF\not\supset \mathrm{B},
\end{cases}
\quad \CCF\subcl M\beta X,
$$
is in $M_*(M_*X)$. Then $\mu_M\beta X(C)(\Cl_{\beta X}A)=1$, but
there is no compact subset $K\subset A$ such that $c_{\CCB}(K)\ne
0$ for all $\CCB\in\CCG$, therefore $\mu_M\beta X(C)(K)=0$ for all
compact $K\subset A=\Cl_{\beta X}A\cap X$, and $\mu_M\beta
X(C)\notin M_*X$.
\end{proof}

It is still unknown to the authors~:
\begin{que}
Does the converse implication hold? Do all locally compact
Hausdorff or (complete) metrizable spaces satisfy the condition of
the previous statement?
\end{que}

\section{Topological properties of the functors $\check
G$, $\hat G$, $\check M$ and $\hat M$}

Recall that a continuous map of topological spaces is \emph{proper}
if the preimage of each compact set under it is compact.
A~\emph{perfect map} is a closed continuous map such that the
preimage of each point is compact. Any perfect map is
proper~\cite{En:GenTop:77}.

>From now on all maps in this section are considered continuous, and
all spaces are Tychonoff if otherwise not specified.

\begin{rem}
We have already seen that properties of the functors $\check M$ and
$\hat M$ are ``parallel'' to properties of the functors $\check G$
and $\hat G$. Therefore in this section we present only
formulations and proofs of statements for $\check M$ and $\hat M$.
\textbf{All} of them are valid also for $\check G$ and $\hat G$,
and it is an easy exercise to simplify the proofs for capacities to
obtain proofs for compact inclusion hyperspaces.
\end{rem}

\begin{stat}
Functors $\check M$ and $\hat M$ preserves the class of injective
maps.
\end{stat}

\begin{proof}
Let $f:X\to Y$ be injective. If $c,c'\in \check MX$ and $A\subcl X$
are such that $c(A)\ne c'(A)$, then $B=\Cl f(A)\in\subcl Y$, and
$\check Mf(c)(B)=c(f^{-1}(B))=c(A)\ne c'(A)=c'(f^{-1}(B))=\check
Mf(c)(B)$, hence $\check Mf(c)\ne \check Mf(c')$, and $\check Mf$
is injective, as well as its restriction $\hat Mf$.
\end{proof}

\begin{stat}
Functors $\check M$ and $\hat M$ preserve the class of closed
embeddings.
\end{stat}

\begin{proof}
Let a map $f:X\to Y$ be a closed embedding (thus a perfect map),
then for the Stone-\v Cech compactification $\beta f:\beta X\to
\beta Y$ the inclusion $\beta f(\beta X\setminus X)\subset \beta
Y\setminus Y$ is valid~\cite{En:GenTop:77}. We know that $M\beta
X(M^*X)\subset M^*Y$, $M\beta X(M_*X)\subset M_*Y$.

Let $c\in M\beta X\setminus M^*X$, then there are $F,G\subcl M\beta
X$ such that $F\cap X\subset G$, but $c(F)>c(G)$. Then $f(F)$ and
$f(G)$ are closed in $\beta Y$, and $f(F)\setminus f(G)\subset
f(\beta X\setminus X)\subset \beta Y\setminus Y$.

The sets $F'=f^{-1}(f(F))$ and $G'=f^{-1}(f(G))$ are closed in
$\beta X$ and satisfy $F'\cap X=F\cap X$, $G'\cap X=G\cap X$, thus
$c(F')=M\beta f(c)(f(F))>c(G')=M\beta f(c)(f(G))$, which implies
$M\beta f(c)\notin M^*Y$. Thus $(M\beta f)^{-1}(M^*Y)=M^*X$, and
the restriction $M\beta f|_{M^*X}:M^*X\to M^*Y$ is perfect,
therefore closed. It is obvious that this restriction is injective,
thus is an embedding. For the maps $M\beta f|_{M^*X}$ and $\check
Mf$ are homeomorphic, the same holds for the latter map.

Now let $c\in M\beta X\setminus M_*X$, i.e. there is $F\subcl
\beta X$ such that $c(F)>\sup \{c(K)\mid K\subset F\cap X\text{ is
compact}\}$. The compact set $F'=\beta f(F)$ is closed in $\Cl
f(X)\subset \beta Y$. Observe that $F=(\beta f)^{-1}(F')$ and
obtain:
\begin{gather*}
\sup \{M\beta f(c)(L)\mid L\subset F'\cap Y\text{ is compact}\}
=
\\
\sup \{c((\beta f)^{-1}(L))\mid L\subset F'\cap Y\text{ is
compact}\}
\le
\\
\sup \{c(K)\mid K\subset F\cap X\text{ is compact}\}
<c(F)=M\beta f(c)(F'),
\end{gather*}
and $M\beta f(c)\notin M_*X$. The rest of the proof is analogous to
the previous case.
\end{proof}

It allows for a closed subspace $X_0\subset X$ to identify $\check
MX_0$ and $\hat MX_0$ with the images of the map $\check Mi$ and
$\hat Mi$, with $i:X_0\hookrightarrow\to X$ being the embedding.

We say that a functor $F$ in $\Tych$ \emph{preserves intersections
(of closed sets)} if for any space $X$ and a family
$(i_\alpha:X_\alpha\hookrightarrow X)$ of (closed) embeddings the
equality $\bigcap_\alpha FX_\alpha=FX_0$ holds, i.e.
$\bigcap_\alpha Fi_\alpha(X_\alpha)=Fi_0(X_0)$, where $i_0$ is the
embedding of $X_0=\bigcap_\alpha X_\alpha$ into~$X$. This notion is
usually used for functors which preserve (closed) embeddings,
therefore we verify that:

\begin{stat}
Functors $\check M$ and $\hat M$ preserve intersections of closed
sets.
\end{stat}

\begin{proof}
Let $c\in \check MX$ and closed subspaces $X_\alpha\subset X$,
$\alpha\in A$, are such that $c\in \check MX_\alpha$ for all
$\alpha\in A$. Let $2_f^A$ be the set of all non-empty finite
subsets of $A$. It is a directed poset when ordered by inclusion.
For all $F\subcl X$ and $\{\alpha_1,\dots,\alpha_k\}\in 2_f^A$ we
have $c(F)=c(F\cap X_{\alpha_1}\cap\dots\cap X_{\alpha_k})$. The
monotonically decreasing net $(F\cap X_{\alpha_1}\cap\dots\cap
X_{\alpha_k})_{\{\alpha_1,\dots,\alpha_k\}\in 2_f^A}$ converges to
$F\cap X_0$, with $X_0=\bigcap_{\alpha\in A}X_\alpha$. Thus
$c(F)=c(F\cap X_0)$, which implies $c\in \check MX_0$.

The statement for $\hat M$ is obtained as a corollary due to the
following observation: if $X_0\subset X$ is a closed subspace, then
$\hat MX_0=\hat MX\cap \check MX_0$.
\end{proof}

Therefore for each element $c\in \check MX$ there is a least closed
subspace $X_0\subset X$ such that $c\in\check MX_0$. It is called
the \emph{support} of $c$ and denoted $\supp c$.

It is unknown to the author whether the functor $\check M$ preserve
finite or countable intersections.

\begin{stat}
Functor $\hat M$ preserves countable intersections.
\end{stat}

\begin{proof}
Let $c\in \hat MX$ belong to all $\hat MX_n$ for a sequence of
subspaces $X_n\subset X$, $n=1,2,\dots$. If $A\subcl F$, $\eps>0$,
then there is a compactum $K_1\subset A\cap X_1$ such that
$c(K_1)>c(A)-\eps/2$. Then choose a compactum $K_2\subset K_1\cap
X_2$ such that $c(K_2)>c(K_1)-\eps/4$, \dots, a compactum
$K_n\subset K_{n-1}\cap X_n$ such that
$c(K_n)>c(K_{n-1})-\eps/2^n$, etc. The intersection
$K=\bigcap_{n=1}^\infty K_n$ is a compact subset of $A\cap X_0$,
$X_0=\bigcap_{n=1}^\infty X_n$, and $c(K)>C(A)-\eps$. Thus $\sup
\{c(K)\mid K\subset A\cap X_0\text{ is compact}\}=c(A)$ for all $A\subcl
X$, i.e. $c\in \hat MX_0$.
\end{proof}

It is easy to show that $\check M$ and $\hat M$ do not preserve
uncountable intersections.

We say that a functor $F$ in $\Tych$ (or in $\Comp$)
\emph{preserves preimages} if for each continuous map $f:X\to Y$
and a closed subspace $Y_0\subset Y$ the inclusion $Ff(b)\in FY_0$
for $b\in FX$ implies $b\in F(f^{-1}(Y_0))$, or, more formally,
$Ff(b)\in Fj(FY_0)$ implies $b\in Fi(F(f^{-1}(Y_0)))$, where
$i:f^{-1}(Y_0)\hra X$ and $j:Y_0\hra Y$ are the embeddings.

\begin{stat}
Functors $\check M$ and $\hat M$ do not preserve preimages.
\end{stat}

It is sufficient to recall that the capacity functor
$M:\Comp\to\Comp$, being the restriction of the two functors in
question, does not preserve preimages~\cite{ZN:CapFun:08}.

\begin{stat}
Let $f:X\to Y$ is a continuous map such that $f(X)$ is dense in
$Y$. Then $\check Mf(\check MX)$ is dense in $\check MY$, and $\hat
Mf(\hat MX)$ is dense in $\hat MY$.
\end{stat}

\begin{proof}
Let $M_\omega X$ be the set of all capacities on $X$ with finite
support, i.e.
$$
M_\omega X=\bigcup\{MK\mid K\subset X\text{ is finite}\}
$$
Then $M_\omega X\subset \hat MX\subset\check MX$, $\check
Mf(M_\omega X)=M_\omega (f(X))$, and the latter set is dense in
both $\hat MY$ and $\check MY$.
\end{proof}

\section{Subgraphs of capacities on Tychonoff space and fuzzy integrals}

In \cite{ZN:CapFun:08} for each capacity $c$ on a compactum $X$ its
\emph{subgraph} was defined as follows~:
$$
\sub c=\{(F,\alpha)\in \exp X\times I\mid \alpha\le c(F)\}.
$$
Given the subgraph $\sub c$, each capacity $c$ is uniquely
restored~: $c(F)=\max\{\alpha\in I\mid (F,\alpha)\in\sub c\}$ for
each $F\in\exp X$.

Moreover, the map $\sub$ is an embedding $MX\hra \exp(\exp X\times
I)$. Its image consists of all sets $S\subset \exp X\times I$ such
that~\cite{ZN:CapFun:08} the following conditions are satisfied for
all closed nonempty subsets $F$, $G$ of $X$ and all
$\alpha,\beta\in I$~:
\begin{enumerate}
\item
if $(F,\alpha)\in S$, $\alpha\ge \beta$, then $(F,\beta)\in S$;
\item
if $(F,\alpha),(G,\beta)\in S$, then $(F\cup G,\alpha\lor\beta)\in
S$;
\item
$S\supset \exp X\times\{0\}\cup \{X\}\times I$;
\item
$S$ is closed.
\end{enumerate}
The topology on the subspace $\sub(MX)\subset \exp(\exp X\times I)$
can be equivalently determined by the subbase which consists of all
sets of the form
$$
V_+(U,\alpha)=\{S\in \sub(MX)\mid \text{there is }(F,\beta)\in S,
F\subset U, \beta>\alpha\}
$$
for all open $U\subset X$, $\alpha\in I$, and of the form
$$
V_-(F,\alpha)=\{S\in \sub(MX)\mid \beta<\alpha
\text{ for all }(F,\beta)\in S\}
$$
for all closed $F\subset X$, $\alpha\in I$.

Let the subgraph of a~$\tau$-smooth capacity $c$ on a Tychonoff
space $X$ be defined by the same formula at the beginning of the
section. Consider the intersection $\sub c\cap (\exp
X\times\{\alpha\})$. It is equal to $S_\alpha(c)\times \{\alpha\}$,
with $S_\alpha(c)=\{F\in\exp X\mid c(F)\ge\alpha\}$. The latter set
is called the $\alpha$-\emph{section}~\cite{ZN:CapFun:08} of the
capacity $c$ and is a~compact inclusion hyperspace for each
$\alpha>0$. Of course, $S_0(c)=\exp X$ is not compact if $X$ is not
compact. If $0\le\alpha<\beta\le 1$, then $S_\alpha(c)\supset
S_\beta$, and $S_\beta(c)=\bigcup_{0\le\alpha<\beta}S_\alpha(c)$.

We present necessary and sufficient conditions for a set
$S\subset\exp X\times I$ to be the subgraph of some capacity $c\in
\hat MX$.

\begin{stat}
Let $X$ be a Tychonoff space. A set $S\subset\exp X\times I$ is a
subgraph of a~$\tau$-smooth capacity on $X$ if and only if the
following conditions are satisfied for all closed nonempty subsets
$F$, $G$ of $X$ and all $\alpha,\beta\in I$~:
\begin{enumerate}
\item
if $(F,\alpha)\in S$, $\alpha\ge \beta$, then $(F,\beta)\in S$;
\item
if $(F,\alpha),(G,\beta)\in S$, then $(F\cup G,\alpha\lor\beta)\in
S$;
\item
$S\supset \exp X\times\{0\}\cup \{X\}\times I$;
\item
$S\cap(\exp X\times [\gamma;1])$ is compact in $\expl X\times I$
for all $\gamma\in (0;1]$.
\end{enumerate}
Such $S$ is closed in $\exp X\times I$.
\end{stat}

\begin{proof}
Let $c\in\check MX$ and $S=\sub c$. It is easy to see that $S$
satisfies (1)--(3). To show that $S\cap(\exp X\times [\gamma;1])$
is compact, assume that it is covered by subbase elements
$U_i^-\times (a_i;b_i)$, $U_i\subop X$, $i\in\CCI$. For any
$\alpha\in[\gamma;1]$ the intersection $S\cap
(\exp\times\{\alpha\})=S_\alpha(c)\times
\{\alpha\}$ is compact and covered by $U_i^-\times (a_i;b_i)$
for those $i\in\CCI$ that $(a_i,b_i)\ni \alpha$. Therefore there is
a finite subcover $U_{i_1}^-$, \dots, $U_{i_k}^-$ of $S_\alpha(c)$,
$\max\{a_{i_1},\dots,a_{i_k}\}<\alpha<\min\{b_{i_1},\dots,b_{i_k}\}$.
When $a\nearrow \alpha$, the compact set $S_a(c)$ decreases to
$S_\alpha(c)$, thus there is
$a\in(\max\{a_{i_1},\dots,a_{i_k}\};\alpha)$ such that
$S_a(c)\subset U_{i_1}^-\cup \dots\cup U_{i_k}^-$. If we denote
$b=\min\{b_{i_1},\dots,b_{i_k}\}$, we obtain that for each
$\alpha\in[\gamma;1]$ there is an interval $(a,b)\ni\alpha$ such
that $S\cap(\exp X\times (a,b))$ is covered by a finite number of
sets $U_i^-\times (a_i,b_i)$. For $[\gamma,1]$ is compact, we infer
that there is a finite subcover of the whole set $S\cap(\exp
X\times [\gamma;1])$, thus (4) holds.

Now let a set $S\subset \exp X\times I$ satisfy (1)--(4), and let
$S_\alpha=\pr_1(S\cap(\exp X\times I))$ for all $\alpha\in I$. By
(1) $S_\alpha\supset S_\beta$ whenever $a<\beta$. Assume
$S_\beta\ne\bigcap_{0<\alpha<\beta}S_\alpha$ for some $\beta\in
(0;1]$, i.e. there is $F\in\exp X$ such that $F\in S_\alpha$ for
all $\alpha\in (0;\beta)$, but $F\notin S_\beta$. Then the sets
$(X\setminus F)^-\times I$ and $\exp X\times [0;\alpha)$, with
$\alpha\in (0;\beta)$, form an open cover of the set $S\cap (\exp
X\times [\beta/2;1])$ for which there is no finite subcover, which
contradicts to compactness. Thus
$S_\beta=\bigcap_{0<\alpha<\beta}S_\alpha$. It implies that for
$(F,\beta)\notin S$, i.e. $F\notin S_\beta$, there is $\alpha\in
(0;\beta)$ such that $F\notin S_\alpha$. The set $S_\alpha$ is a
compact inclusion hyperspace, thus is closed in $\exp X$. Then
$(\exp X\setminus S_\alpha)\times (\alpha;1]$ is an open
neighborhood of $(F,\beta)$ which does not intersect $S$, hence $S$
is closed in $\exp X\times I$.

For each $F\in\exp X$ we put $c(F)=\max\{\alpha\mid (F,\alpha)\in
S\}$. It is straightforward to verify that $c$ is a $\tau$-smooth
capacity such that $\sub c=S$.
\end{proof}

\begin{stat}
Let $\psi:\exp X\times I\to I$ be a continuous function such that~:
\begin{enumerate}
\item
$\psi$ in antitone in the first argument and isotone in the second
one;
\item
$\psi(F,\alpha)$ uniformly converges to $0$ as $\alpha\to 0$.
\end{enumerate}
Then the correspondence $\Psi:c\mapsto \max\{\psi(F,c(F))\mid F\in
\exp X\}$ is a well defined continuous function $\check MX\to I$.
\end{stat}

\begin{proof}
Let $S=\sub c$. Observe that $\Psi$ can be equivalently defined as
$\Psi(c)=\max\{\psi(F,\alpha)\mid (F,\alpha)\in S\}$. The function
$\psi:\expl X\times I\to I$ is upper semicontinuous, and $\psi(S)$
is either $\{0\}$ or equal to $\psi(S\cap(\exp X\times
[\gamma;1]))$ for some $\gamma\in (0;1)$. Hence $\psi(S)$ is a
compact subset of $I$, therefore contains a greatest element, and
use of ``$\max$'' in the definition of $\Psi$ is legal.

Assume that $\Psi(c)<b$ for some $b\in I$. We take some $a\in
(\Psi(c);b)$. There exists $\gamma\in I$ such that
$\psi(F,\alpha)<a$ for all $\alpha\in [0;\gamma)$, $F\in\exp X$. If
$(F,\alpha)\in S$, $\alpha\ge\gamma$, then there is a neighborhood
$\CCV=\langle U_0,U_1,\dots,U_k\rangle\times (u,v)\ni (F,\alpha)$
such that $U_1\cup\dots\cup U_k$ is completely separated from
$X\setminus U_0$, and $\psi(G,\beta)<a$ for all $(G,\beta)\in\CCV$.
The inequality $\psi(G,\beta)<a$ holds also for all $(G,\beta)\in
\langle X,U_1,\dots,U_k\rangle\times [0,v)$. Thus we obtain a
cover of $S\cap (\exp X\times [\gamma;1])$ by open sets in $\expl
X\times I$, and there is a finite subcover by sets $\langle
X,U_1^l,\dots,U_{k_l}^l\rangle\times [0,v_l)$, $1\le l\le n$. We
may assume $0<v_1\le v_2\le\dots\le v_n>1$. It is routine but
straightforward to verify that $c$ is in an open neighborhood
\begin{multline*}
\CCU=
\bigcap\{
O_-(X\setminus(U^{m+1}_{j_{m+1}}\cup
\dots\cup
U^n_{j_n}),v_m)
\mid
1\le m< n,
\\
1\le j_{m+1}\le k_{m+1},
\dots,
1\le j_n    \le k_n
\},
\end{multline*}
and for each capacity $c'\in\CCU$ the set $\sub c'\cap[\gamma;1]$
is also covered by the sets
$$
\langle
X,U_1^l,\dots,U_{k_l}^l\rangle\times [0,v_l),\; 1\le l\le n,
$$
therefore
$$
\Psi(c')\le\max\{a,\max\{\psi(F,\alpha)
\mid
(F,\alpha)\in S,\alpha\ge\gamma\}\}
= a<b.
$$
Hence $\Psi$ is upper semicontinuous. To prove lower
semicontinuity, assume that $\Psi(c)>b$ for some $b\in I$. Then
there is $F\in\exp X$ such that $\psi(F,c(F))>b$. By continuity
there are open neighborhood $U\supset F$ and $\gamma\in(0;c(F))$
such that $F$ is completely separated from $X\setminus U$, and for
all $G\in\exp X$, $G$ completely separated from $X\setminus U$,
$\alpha\in I$, $\alpha>\gamma$ the inequality $\psi(G,\alpha)>b$ is
valid. Then $c\in O_+(U,\alpha)$, and for all $c'\in O_+(U,\alpha)$
we have $\Psi(c')>b$.
\end{proof}

The reason to consider such form of $\Psi$ is that not only Sugeno
integral can be represented this way (for
$\psi(F,\alpha)=\inf\{\phi(x)\mid x\in F\}\land
\alpha$), but a whole class of fuzzy integrals obtained by
replacement of ``$\land$'' by an another ``pseudomultiplication''
$\odot:I\times I\to I$~\cite{BenMeVi:MonSBFun:02}, e.g. by usual
multiplication or the operation $h(a,b)=a+b-ab$. The latter
statement provides the continuity of a fuzzy integral with respect
to a capacity on a Tychonoff space, provided ``$\odot$'' is
continuous, isotone in the both variables and uniformly converges
to $0$ as the second argument tends to $0$ (which is not the case
for the $h$ given above).


\end{document}